\documentclass[12pt,reqno]{amsart}
\usepackage{amscd}
\usepackage{amsfonts}
\usepackage{amsmath}
\usepackage{amssymb}
\usepackage{physics}
\usepackage{setspace}	
\usepackage{wrapfig}
\usepackage{amsthm}
\usepackage{fancyhdr}
\usepackage{latexsym}
\usepackage[colorlinks=true, pdfstartview=FitV, linkcolor=blue,
            citecolor=blue, urlcolor=blue]{hyperref}
\usepackage{enumerate}  
\usepackage{mathtools}            
\usepackage{upgreek}
\usepackage{extarrows}

\usepackage{cancel}

\input epsf
\input texdraw
\input txdtools.tex
\input xy
\xyoption{all}

\usepackage{tikz}
\usepackage{tikz-cd}

\usetikzlibrary{patterns}
\usetikzlibrary{calc,matrix, arrows}
\usetikzlibrary{snakes}
\usetikzlibrary{arrows.meta}
\usepackage{subfigure}
\usetikzlibrary{decorations.pathreplacing}
\tikzstyle{level 1}=[level distance=3.5cm, sibling distance=8cm]
\tikzstyle{level 2}=[level distance=3.5cm, sibling distance=3cm]
\tikzstyle{level 3}=[level distance=5cm, sibling distance=1.4cm]

\usepackage{pgfplots}
\pgfplotsset{%
every x tick/.style={black, thick},
every y tick/.style={black, thick},
every tick label/.append style = {font=\footnotesize},
every axis label/.append style = {font=\footnotesize},
compat=1.12
  }
  \usepackage{blindtext}

\usepackage{color}

\definecolor{Red}{rgb}{1.00, 0.00, 0.00}
\definecolor{DarkGreen}{rgb}{0.00, 1.00, 0.00}
\definecolor{Blue}{rgb}{0.00, 0.00, 1.00}
\definecolor{Cyan}{rgb}{0.00, 1.00, 1.00}
\definecolor{Magenta}{rgb}{1.00, 0.00, 1.00}
\definecolor{DeepSkyBlue}{rgb}{0.00, 0.75, 1.00}
\definecolor{DarkGreen}{rgb}{0.00, 0.39, 0.00}
\definecolor{dgreen}{RGB}{0,200,100}
\definecolor{ddgreen}{RGB}{0,170,0}
\definecolor{SpringGreen}{rgb}{0.00, 1.00, 0.50}
\definecolor{DarkOrange}{rgb}{1.00, 0.55, 0.00}
\definecolor{OrangeRed}{rgb}{1.00, 0.27, 0.00}
\definecolor{DeepPink}{rgb}{1.00, 0.08, 0.57}
\definecolor{DarkViolet}{rgb}{0.58, 0.00, 0.82}
\definecolor{SaddleBrown}{rgb}{0.54, 0.27, 0.07}
\definecolor{Black}{rgb}{0.00, 0.00, 0.00}
\definecolor{dark-magenta}{rgb}{.5,0,.5}
\definecolor{myblack}{rgb}{0,0,0}
\definecolor{darkgray}{gray}{0.5}
\definecolor{lightgray}{gray}{0.75}

\newcommand{\p}{\partial}
\newcommand{\R}{\mathbb{R}}

\usepackage{scalerel}
\usepackage{stackengine}
\stackMath
\def\hatgap{-5.5pt}
\def\subdown{-3.3pt}
\newcommand\what[2][]{%
\renewcommand\stackalignment{l}%
\stackon[\hatgap]{#2}{%
\stretchto{%
    \scalerel*[\widthof{$#2$}]{\kern-.6pt\text{\textasciicircum}\kern-1.1pt}%
    {\rule[-0.8\textheight]{1ex}{\textheight}}
}{2ex}
_{\smash{\belowbaseline[\subdown]{\scriptscriptstyle#1}}}%
}}

\newcommand{\rr}{\mathbb{R}}

\renewcommand{\p}{\partial}

\makeatletter
\def\mathcolor#1#{\@mathcolor{#1}}
\def\@mathcolor#1#2#3{%
  \protect\leavevmode
  \begingroup
    \color#1{#2}#3%
  \endgroup
}
\makeatother

\DeclareMathSizes{12}{12}{7}{5}

\def\refer #1\par{\noindent\hangindent=\parindent\hangafter=1 #1\par}

\theoremstyle{plain}  
\newtheorem{theorem}{Theorem}[section]

\newtheorem{lemma}{Lemma}[section]

\usepackage[breakable, theorems, skins]{tcolorbox}
\tcbset{enhanced}

\theoremstyle{definition}

\newtheorem{remark}{Remark}[section]

\newenvironment{Proof}[1][\proofname]
{\proof[\textnormal{\textbf{#1.}}]}{\endproof}
\newcommand{\bp}{\begin{Proof}}
\newcommand{\ep}{\end{Proof}}

\numberwithin{figure}{section}
\numberwithin{equation}{section}

\usepackage[framemethod=TikZ]{mdframed}

\def\ra{\right>}
\def\la{\left<}
\def\lm{\left\|}
\def\rm{\right\|}

\begin{document}

\title[Improved lower bound on the radius of analyticity]{Improved lower bound for the radius of analyticity 
			for the modified KdV equation }
\author{ Renata O. Figueira and Mahendra  Panthee}
\address{Department of Mathematics, University of Campinas\\13083-859 Campinas, SP, Brazil}
\thanks{This work is partially supported by FAPESP, Brazil.}

\maketitle
\onehalfspacing

\begin{abstract}
We investigate the initial value problem (IVP) associated to the modified Korteweg-de Vries equation (mKdV) in the defocusing scenario:
\begin{equation*}
\left\{\begin{array}{l}
\partial_t u+ \partial_x^3u-u^2\p_x(u) 
=
0,
\quad x,t\in\mathbb{R}, \\
u(x,0)
=
u_0(x),
\end{array}\right.
\end{equation*}
where $u$ is a real valued function and the initial data $u_0$ is analytic on $\rr$ and has uniform radius of analyticity $\sigma_0$ in the spatial variable.
 It is well-known  that the solution $u$ preserves its analyticity with the same radius $\sigma_0$ for at least some time span $0<T_0\le 1$.
This local result was obtained in \cite{FP23} by proving a trilinear estimate in the Gevrey spaces $G^{\sigma, s}$, $s\geq \frac14$. Global in time behaviour of the solution and algebraic lower bound of the evolution of the radius of analyticity was also studied in authors' earlier works in \cite{FP23} and \cite{FP24} by constructing almost conserved quantities in the classical  Gevrey  space  with $H^1$ and $H^2$ levels of Sobolev regularities.
The present study aims to construct a new almost conservation law in the Gevrey space defined with a weight function  $\cosh(\sigma|\xi|)$ and use it   demonstrate that the local solution $u$ extends globally in time, and the radius of spatial analyticity is bounded from below by $c T^{-\frac{1}{2}}$, for any time $T\geq T_0$. 
The outcome of this paper represents an improvement on the one achieved by the authors' previous work in  \cite{FP24}.

\end{abstract}
\vskip 0.3cm

 
 \vspace{0.2cm}

\vspace{0.5cm}
\section{Introduction}

Continuing the research started in \cite{FP23} and \cite{FP24}, the main subject of study in this work is the initial value problem (IVP) for the modified Korteweg-de Vries (mKdV) equation
\begin{equation}\label{mKdV-IVP}
\left\{\begin{array}{l}
\partial_t u+ \partial_x^3u+\mu u^2\p_x(u) 
=
0,
\quad x,t\in\rr, \\
u(x,0)
=
u_0(x),
\end{array}\right.
\end{equation}
where $\mu=\pm 1$, $u$ is a real-valued function, and $u_0$ belongs to the Gevrey class $G^{\sigma,s}(\mathbb{R})$.

The mKdV equation \cite{M} emerges as an extension of the renowned Korteweg-de Vries (KdV) equation \cite{KdV}, distinguished by its focusing case for $\mu = 1$ and defocusing for $\mu = -1$.
Widely prevalent across diverse physical scenarios, the mKdV equation finds application in various domains. Notable examples include its role in plasma wave propagation \cite{D}, dynamics of traffic flow \cite{W}, studies in fluid mechanics \cite{J} and investigations into nonlinear optics \cite{HK}.

In the recent time, the IVP \eqref{mKdV-IVP} associated to the mKdV equation with given data in the Gevrey spaces  $G^{\sigma,s}(\mathbb{R})$ has attracted attention, see for example \cite{BGK}, \cite{FP23} and \cite{FP24}. For $\sigma>0$ and $s\in \mathbb{R}$, the Gevrey space $G^{\sigma,s}(\mathbb{R})$ is generally defined as
$$
G^{\sigma,s}(\mathbb{R})
:=
\left\{ f\in L^2(\mathbb{R});\;
\|f\|_{G^{\sigma,s}}
=
\Big(\int \la\xi\ra^{2s}e^{2\sigma |\xi|}|\widehat{f}(\xi)|^2 d{\xi}\Big)^\frac 12
<
\infty\right\},
$$
where $\la\xi\ra= (1+|\xi|)$ and $\widehat{f}$ denotes the Fourier transform given by
$$
\widehat{f}(\xi)
=\mathcal{F}(f)(\xi)
=
c\int e^{-ix\xi}f(x)d x.
$$
This space measures the regularity of the initial data and is, in some sense, a generalization of the classical $L^2$-based Sobolev space because for $\sigma=0$, $G^{\sigma, s}(\R)$ simply becomes $H^s(\R)$.

Derived from the well-known Paley-Wiener theorems, it is established that a function $f\in G^{\sigma,s}(\mathbb{R})$ is analytic and possesses a holomorphic extension $\tilde{f}$ defined on the strip $S_\sigma=\{x+iy; |y|<\sigma\}$, satisying $\sup_{|y|<\sigma} \|\tilde{f}(\cdot+iy)\|_{H^s}<\infty$ (we refer to page 174 in \cite{Ka} for details).
The parameter $\sigma>0$ is commonly referred to as the uniform radius of analyticity, as it determines the width of the strip where the function in $G^{\sigma,s}(\mathbb{R})$ can be holomorphically extended.

The well-posedness issues of the IVP \eqref{mKdV-IVP} with given data in the classical Sobolev spaces $H^s(\R)$ is well understood. For the sharp local well-posedness result, we refer to \cite{KPV93}, where the authors used the smoothing effect of the associated linear group combined with Strichartz and maximal function estimates to accomplish the work. This result was later reproduced in \cite{Tao} using trilinear estimates in the Bourgain's space framework, which is defined as the completion of the Schwartz space with respect to the norm
\begin{equation}\label{Xsb}
\|f\|_{X^{s,b}}
=
\big\|\la\xi\ra^{s}\la\tau-\xi^{3}\ra^{b}\widehat{f}(\xi,\tau)\big\|_{L^2_{\xi,\tau}},
   \end{equation}
for $s,b\in\rr$. Bourgain's space $X^{s,b}$, whose norm \eqref{Xsb} is defined using the restriction of the Fourier transform, has proved to be very convenient for dealing with the well-posedness issues for dispersive equations with low regularity  data.  Using this framework, the sharp global well-posedness result for the IVP \eqref{mKdV-IVP} with initial data in the classical Sobolev spaces $H^s(\R)$ was obtained in \cite{CKSTT}.

Concerning the IVP \eqref{mKdV-IVP} with initial data $u_0$  in the Gevrey spaces $G^{\sigma_0,s}(\mathbb{R})$ for a fixed $\sigma_0>0$, the following natural questions arise: 

\begin{itemize}
\item Is it possible to obtain a local well-posedness result and maintain the radius of analyticity during the local existence time?
\item Is it possible to extend the local solution $u(t)$ to any time interval $[0, T]$, and what is the behaviour of the radius of analyticity, say $\sigma(t)$, as time progresses? If $\sigma(t)$ decreases as time evolves, is it possible to find a lower bound?
\end{itemize}

In recent time, these sort of questions are widely studied for dispersive equations. For example \cite{BFH, BGK, DMT, FP23, FP24, GTB, GK-1, SS, ST, T, TTB} are a few notable works. For the IVP \eqref{mKdV-IVP} with initial data in $G^{\sigma, s}(\R)$, the first local well-posedness result was obtained in \cite{GK-1} whenever $s>\frac32$. The local well-posedness result proved in \cite{GK-1} ensures the existence of a finite time $T_0>0$ such that $\sigma(t)=\sigma_0$ for all $|t|\leq T_0$, indicating that the radius of analyticity of the solution remains the same as the initial one for at least in small time interval. 
Now, a natural question that one may raise, is it possible to extend the local solution globally in time and what happens to $\sigma(t)$ for $|t|>T_0$? The global well-posedness issue was addressed in \cite{BGK}, where the authors  proved that the local solution can be extended globally in time to any given interval $[0, T]$ and also obtained an algebraic lower bound $cT^{-12}$ for the evolution of the radius of analyticity. The next question one may naturally ask is whether the lower bound obtained in \cite{BGK} can be improved, in the sense that the decay rate of the radius of analyticity could be slower than the one obtained in \cite{BGK}, allowing the solution $u(t)$ to remain analytic in a larger strip for a longer time.

Motivated by the question raised in the previous paragraph, the authors in \cite{FP23} considered the IVP  \eqref{mKdV-IVP} with initial data in $G^{\sigma, s}(\R)$ and proved the local well-posedness result for $s\geq \frac14$ and also obtained $cT^{-\frac43}$ as an algebraic lower bound for evolution of the radius of analyticity while extending local solution globally in time, offering a significant improvement in lower bounds for the radius of analyticity obtained in \cite{BGK}. This lower bound was further improved to $cT^{-1}$ in the authors' subsequent work \cite{FP24}.

The technique used in \cite{FP23} and \cite{FP24}, which is also employed here in this work, essentially consists in proving an almost conservation law of the form
\begin{equation}
\label{ACL}
\sup\limits_{t\in[0,T_0]}
A_\sigma(t)
\le
A_\sigma(0) +C\sigma^\theta A_\sigma(0)^n, 
\;\; \theta>0,
\end{equation}
for all $0<\sigma\le \sigma_0$, where $C$ is a positive constant, $n\in\mathbb{N}$ and $A_\sigma(t)$ is a quantity related to the conservation laws satisfied by the flow of \eqref{mKdV-IVP}, and consequently, with the norm of the local solution at time $t\in [0, T_0]$, i.e., $\|u(\cdot,t)\|_{G^{\sigma,s}}$. This idea of almost conservation law was first introduced in \cite{ST} and has since been utilized by several authors, see for example \cite{SS}, \cite{FP23}, \cite{FP24} and references therein. With the almost conservation law of the form \eqref{ACL} at hand,  it is possible to extend the local solution $u$ globally in time  by decomposing any time interval $[0,\,T]$ into short subintervals and iterating the local result in each subinterval. During this iteration process, a restriction emerges that provides a lower bound for the radius of analyticity $\sigma(T)$ as described below
$$
u\in C([-T,T]; G^{\sigma(T),s}(\rr)), \;\;\text{ with }\;\; \sigma(T)\ge \frac{c}{T^\frac 1\theta},
$$
where $\theta$ corresponds precisely to the value appearing in \eqref{ACL} and $c>0$ is a positive constant independent of $T$.
It is noteworthy that a larger value of $\theta$ yields a better lower bound for $\sigma(t)$, signifying a slower decay of the radius of analyticity as time goes to infinity.

In our previous work \cite{FP23}, the following almost conservation law at level $G^{\sigma, 1}(\mathbb{R})$ in the defocusing case ($\mu = -1$) was proven
\begin{equation}
\label{ACL1}
\sup\limits_{t\in[0,T_0]}
A_{I,\sigma}(t)
\leq
A_{I,\sigma}(0) + C\sigma^{\theta} A_{I,\sigma}(0)^2 \big(1+A_{I,\sigma}(0)\big),
\qquad \theta\in\Big[0, \frac34\Big],
\end{equation}
where $A_{I,\sigma}(t)$ is defined by
\begin{equation*}
A_{I,\sigma}(t)
=
\|u(t)\|^2_{G^{\sigma,1}}-\frac{\mu}{6} \|e^{\sigma |D_x|}u\|_{L^4_x}^4,
\end{equation*}
thus obtaining $cT^{-\frac43}$ as a lower bound for the radius of analyticity.
While in the subsequent work \cite{FP24}, also in the defocusing case, an improvement in the interval to which the exponent $\theta$ belongs was achieved by proving the following conservation law at the level $G^{\sigma,2}(\mathbb{R})$
\begin{equation}
\label{ACL2}
\sup\limits_{t\in[0,T_0]}
A_{II,\sigma}(t)
\leq
A_{II,\sigma}(0) + C\sigma^\theta A_{II,\sigma}(0)^2 \big(1+A_{II,\sigma}(0)+A_{II,\sigma}(0)^2\big),
\qquad \theta\in [0,1],
\end{equation}
where $A_{II,\sigma}(t)$ is given by
\begin{equation}
\label{AII}
\begin{split}
A_{II,\sigma}(t)
&=
\|e^{\sigma |D_x|}u(t)\|^2_{L_x^2}+\|\p_x(e^{\sigma |D_x|}u(t))\|^2_{L_x^2} +\|\p_x^2(e^{\sigma |D_x|}u(t))\|^2_{L_x^2}\\
& \quad-
\frac{\mu}{6}\|e^{\sigma |D_x|}u(t)\|_{L^4_x}^4
-
\frac{5\mu}{3}
\|(e^{\sigma |D_x|}u)\cdot\p_x(e^{\sigma |D_x|}u)(t)\|_{L^2_x}^2
+
\frac{1}{18}
\|e^{\sigma |D_x|}u(t)\|_{L_x^6}^6,
\end{split}
\end{equation}
and consequently, obtaining $cT^{-1}$ as a lower bound for the radius of analyticity of the solution.

Since the mKdV equation has infinitely many conserved quantities, the natural path to follow would be to consider the almost conservation laws at  higher levels of regularity.
However, a crucial observation must be made: both proofs of the inequalities \eqref{ACL1} and \eqref{ACL2} rely on a particular property of the exponential function
\begin{equation}
\label{exp-est}
e^{\sigma |\xi|}-1
\leq
(\sigma |\xi|)^\theta e^{\sigma |\xi|},
\quad \theta\in[0,1].
\end{equation}
The presence of the the exponential weight  $e^{\sigma |\xi|}$ in the definition of the Gevrey space  $G^{\sigma,s}(\rr)$ naturally leads to use estimate of the form \eqref{exp-est}. Therefore, due to the necessity of inequality \eqref{exp-est} in the technique presented in the proofs of \eqref{ACL1} and \eqref{ACL2}, the highest achievable value for $\theta$ is $\theta=1$, resulting in $cT^{-1}$ as the best lower bound for the radius of analyticity.  This is the reason for considering the almost conserved quantities only up to level $G^{\sigma,2}$ using this approach with the exponential weight in the norm of $G^{\sigma,s}(\rr)$. Taking this observation in mind, to get a better lower bound for the evolution of the radius of analyticity one needs to find some other alternative to avoid the use of the estimate \eqref{exp-est} in the form as it is.

Motivated from the recent contributions in \cite{DMT}, \cite{GTB} and \cite{TTB}, our current investigation considers a modified Gevrey norm by replacing the exponential weight $e^{\sigma |\xi|}$ in $\| \cdot\|_{G^{\sigma,s}}$ with a hyperbolic cosine weight $\cosh(\sigma |\xi|)$, defined as follows
\begin{equation}
\| f\|_{H^{\sigma,s}}
=
\Big(\int \la\xi\ra^{2s}\cosh^2(\sigma |\xi|)|\widehat{f}(\xi)|^2 d{\xi}\Big)^\frac 12.
\end{equation}
The norms $\| \cdot\|_{G^{\sigma,s}}$ and $\| \cdot\|_{H^{\sigma,s}}$ are equivalent due to the trivial estimate
\begin{equation}
\label{cosh-exp-est}
\frac 12 e^{\sigma |\xi|} 
\le
\cosh (\sigma |\xi|)
\le
e^{\sigma |\xi|}.
\end{equation}
The use of the hyperbolic cosine  weight in the Gevrey norm offers an advantage as the function $\cosh(\sigma|\xi|)$ complies with the estimate
\begin{equation}\label{cosh-1}
\cosh(\sigma|\xi|)-1\leq (\sigma|\xi|)^{2\theta}\cosh(\sigma|\xi|), \qquad \theta\in [0,1].
\end{equation}
Observe that, considering $\theta =1$ in \eqref{cosh-1} it is possible  to achieve an exponent $2$ for $\sigma$ in the almost conservation law, which translates in $cT^{-\frac 12}$ as the new lower bound for the radius of analyticity (see Theorems \ref{ACL-mKdV-thm} and \ref{global-mKdV-thm}  below).
 
The analytic Bourgain space with  hyperbolic cosine as the weight function $X^{\sigma,s,b}$ related to the mKdV equation are used to get the results of this work, whose norm is defined as
\begin{equation}
\label{Xsb-norm}
\|u\|_{X^{\sigma,s,b}}
=
\big\|\cosh(\sigma|\xi|)\la\xi\ra^{s}\la\tau-\xi^{3}\ra^{b}\widehat{u}(\xi,\tau)\big\|_{L^2_{\xi,\tau}},
   \end{equation}
for $s,b\in\rr$ and $\sigma >0$. 
Also, for $T>0$ we consider the restricted Bourgain's space $X^{\sigma,s,b}_{T}$
with norm  given by
\begin{equation}
\label{sb-restrict-norm}
\| u\|_{X^{\sigma,s,b}_{T}}
=
\inf\limits_{\tilde{u}\in X^{\sigma,s,b}}\big\{
\| \tilde{u}\|_{X^{\sigma,s,b}};\;\tilde{u}(x,t)
=
u(x,t)
\;\;\mbox{on}\;\; \mathbb{R}\times[-T,T]
\big\}.
\end{equation}
Similarly, we define $\| u\|_{X^{s,b}_{T}}$ considering $\sigma =0$ in \eqref{sb-restrict-norm}.
Considering the hyperbolic cosine as the weight function in the space measuring the regularity of the initial data, it is reasonable to extend this consideration to the Bourgain spaces.

Again, in the light of \eqref{cosh-exp-est}, one observes that $\|\cdot\|_{X^{\sigma,s,b}}$ and  $\|\cdot\|_{Y^{\sigma,s,b}}$ are equivalent, where 
\begin{equation}
\label{Ysb-norm}
\|u\|_{Y^{\sigma,s,b}}
=
\big\| e^{\sigma|\xi|}\la\xi\ra^{s}\la\tau-\xi^{3}\ra^{b}\widehat{u}(\xi,\tau)\big\|_{L^2_{\xi,\tau}}
 \end{equation}
is the norm of  the analytic Bourgain's space $Y^{\sigma,s,b}$ with exponential weight function.

Using that the norms   $\| \cdot\|_{G^{\sigma,s}}$ and $\| \cdot\|_{H^{\sigma,s}}$, and  $\|\cdot\|_{X^{\sigma,s,b}}$ and  $\|\cdot\|_{Y^{\sigma,s,b}}$ are equivalent, with the procedure used in \cite{FP23} one can easily prove the following local well-posedness result.

\begin{theorem}\label{lwp-mKdV-thm}
Let $\sigma_0>0$, $b>\frac12$ and $s\ge \frac14$. For each $u_0\in H^{\sigma_0,s}(\rr)$ there exists a time 
\begin{equation}\label{lifetime}
T_0=T_0(\|u_0\|_{H^{\sigma_0,s}})=\frac{c_0}{(1+\|u_0\|_{H^{\sigma_0,s}}^2)^a},\qquad c_0>0, \quad a>1,
\end{equation}
 such that the IVP \eqref{mKdV-IVP} admits a unique solution $u$ in $ C([-T_0,T_0] ; H^{\sigma_0,s}(\rr))\cap X_{T_0}^{\sigma_0,s,b}$ satisfying
\begin{equation}\label{bound.u}
\|u\|_{X^{\sigma_0,s,b}_{T_0}} 
\le 
C\|u_0\|_{H^{\sigma_0,s}}.
\end{equation}
 
\end{theorem}

In order to extend the local solution obtained in Theorem \ref{lwp-mKdV-thm} globally in time and obtain the lower bound for the evolution of the radius of analyticity, as mentioned above, we will  derive an almost conserved quantity in the $H^{\sigma, 2}$ space. For this we consider the following quantity in modified Gevrey class $H^{\sigma,s}(\rr)$ 
\begin{equation}
\label{A}
\begin{split}
A_\sigma(t)
&:=
\|\cosh(\sigma |D_x|)u(t)\|^2_{L_x^2}+\|\p_x[\cosh(\sigma |D_x|)u(t)]\|^2_{L_x^2} +\|\p_x^2[\cosh(\sigma |D_x|)u(t)]\|^2_{L_x^2}\\
& \quad-
\frac{\mu}{6}\|\cosh(\sigma |D_x|)u(t)\|_{L^4_x}^4
-
\frac{5\mu}{3}
\|[\cosh(\sigma |D_x|)u]\cdot\p_x[\cosh(\sigma |D_x|)u](t)\|_{L^2_x}^2\\
&\quad +
\frac{1}{18}
\|\cosh(\sigma |D_x|)u(t)\|_{L_x^6}^6,
\end{split}
\end{equation}
which comes from \eqref{AII} by replacing $e^{\sigma |D_x|}$ by $\cosh(\sigma |D_x|)$.

With this framework at hand, in the following theorem we will derive an almost conservation law that plays a crucial role to prove the main result of this work.
\begin{theorem}
\label{ACL-mKdV-thm}
Let $\sigma >0$, $u_0\in H^{\sigma,2}(\rr)$ and $u\in C([-T_0,T_0]; H^{\sigma,2}(\rr))$ be the local solution of the IVP \eqref{mKdV-IVP} given by Theorem \ref{lwp-mKdV-thm}.
Let $A_\sigma(t)$ be as defined in \eqref{A}. Then for some $b>1/2$ and  for all $t\in [0,T_0]$, we have
\begin{equation}
\label{ACL-mKdV}
A_\sigma(t)
\le
A_\sigma(0) + C\sigma^2 \|u\|^4_{X^{\sigma,2,b}_{T_0}} \big(1+\|u\|^2_{X^{\sigma,2,b}_{T_0}}+\|u\|^4_{X^{\sigma,2,b}_{T_0}}\big).
\end{equation}
Moreover, considering $\mu=-1$, one has
\begin{equation}
\label{ACL-mKdV-0}
A_\sigma(t)
\le
A_\sigma(0) + C\sigma^2 A_\sigma(0)^2 \big(1+A_\sigma(0)+A_\sigma(0)^2\big).
\end{equation}
\end{theorem}

As observed in \cite{FP24} (see Remark $4.1$ there), there is price to pay when we use the $\cosh(\sigma |\xi|)$ instead of $e^{\sigma |\xi|}$ as the weight function. With the hyperbolic cosine, one needs to absorb two derivatives  while constructing the almost conservation law \eqref{ACL-mKdV}. 
The discovery of how to handle these extra derivatives can be found in the proof of Theorem \ref{ACL-mKdV-thm},  particularly during the derivation of inequality \eqref{est-r4}. 
The key ingredients to obtain \eqref{est-r4} are the following  estimates satisfied by the linear group associated to the linear part of the IVP \eqref{mKdV-IVP} 
\begin{equation}
\label{kpv-1}
\| e^{-t\p_x^3} u_0\|_{L_x^2L^\infty_T}
\le
C\|u_0\|_{H^s}, \text{ for all } T\lesssim 1, \; u_0\in H^s(\rr) \;\text{ and }\; s> \frac 34,
\end{equation}
whose proof can be found in \cite{KPV91-1} (see Corollary 2.9 there)  and
\begin{equation}
\label{kpv-2}
\|\p_x e^{-t\p_x^3} u_0\|_{L_x^\infty L^2_T}
\le
C\|u_0\|_{L^2}, \text{ for all } u_0\in L^2(\rr),
\end{equation}
whose the proofs can be found in \cite{KPV93}, see also \cite{KPV91}. 

Finally, we are in position to state the main result of this work that describes evolution in time and a new lower bound for the radius of analyticity to the solution of the IVP \eqref{mKdV-IVP}.

\begin{theorem}\label{global-mKdV-thm}
Let $\sigma_0>0$, $s\ge 1/4$, $u_0\in H^{\sigma_0,s}(\rr)$ and $u\in C([-T_0,T_0];H^{\sigma_0,s}(\rr))$ be the local solution for the IVP \eqref{mKdV-IVP} guaranteed by the local well-posedness result presented in Theorem \ref{lwp-mKdV-thm}. Then, in the defocusing case $(\mu = -1)$, for any $T\ge T_0$, the solution $u$ can be extended globally in time and
$$
u\in C\big([-T,T]; H^{\sigma(T),s}(\rr)\big),
\quad\text{with}\quad
\sigma(T)\ge\min\Big\{\sigma_0, \frac{c}{T^{\frac 12}}\Big\},
$$
with $c$ being a positive constant dependent on $s$, $\sigma_0$ and $\|u_0\|_{H^{\sigma_0, s}}$. 
\end{theorem}

\noindent
{\bf Notations:} Throughout the text we will use standard notations that are commonly used in the partial differential equations.  Besides $\widehat{f}$, the notation $\mathcal{F}_xf$ is used to denote the partial Fourier transform with respect to spatial variable $x$.  We use $c$ or $C$ to denote various  constants whose exact values are immaterial and may
 vary from one line to the next. We use $A\lesssim B$ to denote an estimate
of the form $A\leq cB$ and $A\sim B$ if $A\leq cB$ and $B\leq cA$.

The structure of this paper unfolds as follows. Section \ref{Sec-2} provides fundamental preliminary estimates. In Section \ref{Sec-3}, the almost conservation law is established as stated in Theorem \ref{ACL-mKdV-thm}. Finally, in  Section \ref{Sec-4} the new algebraic lower bound for the radius of analyticity is obtained by proving Theorem \ref{global-mKdV-thm}.
%

%
%
%
%
%
%
%
\section{Preliminaries} 
\label{Sec-2}

In what follows, we present a list of results that will be useful in the proof of the almost conservation law.

\begin{lemma}[Lemma $5$ in \cite{SS}]
\label{est-XT}
Let $\sigma\geq 0$, $s\in\mathbb{R}$, $-1/2<b<1/2$ and $T>0$. 
Then, for any time interval $I\subset [-T,T]$, we have
$$
\lm \chi_{I}u\rm_{X^{\sigma,s,b}}\leq C\lm u\rm_{X_{T}^{\sigma,s,b}},
$$
where $\chi_{I}$ is the characteristic of $I$ and $C$ depends only on $b$.
\end{lemma}

\begin{lemma} Let $b>1/2$, then the following Strichartz's type estimate
\begin{equation}\label{str-1}
\|u\|_{L^p_xL^p_t}
\le
C\|u\|_{X^{0,b}},
\end{equation}
holds  for $p=6,8$.
\end{lemma}

\begin{proof}
The proof of \eqref{str-1} for  $p=8$ follows from  the famous Strichartz estimate for the Airy group obtained in \cite{KPV91} (see also  \cite{Axel-1}). Interpolating this with the trivial estimate $\|w\|_{L^2_{xt}}\leq C \|w\|_{X^{0,0}}$ yields \eqref{str-1} for  $p=6$, see \cite{LG}.
\end{proof}

The following lemma is an immediate consequence of the  Strichartz's estimate  \eqref{str-1} and the generalized H\"older's inequality.

\begin{lemma}
\label{strichartz.lemma}
For all $b>1/2$ and $U_1,U_2,U_3, U_4\in X^{0,b}$, one has
\begin{equation*}
\begin{split}
\|U_1 U_2 U_3\|_{L^{2}_xL_{t}^2}
&\le
C\|U_1\|_{X^{0,b}}\|U_2\|_{X^{0,b}}\|U_3\|_{X^{0,b}},
\\
\|U_1 U_2 U_3 U_4\|_{L^{2}_xL_{t}^2}
&\le
C\|U_1\|_{X^{0,b}}\|U_2\|_{X^{0,b}}\|U_3\|_{X^{0,b}}\|U_4\|_{X^{0,b}}.
\end{split}
\end{equation*} 
\end{lemma}

As mentioned in the introduction, when applying the hyperbolic cosine weight in proving the almost conservation law stated in Theorem \ref{ACL-mKdV-thm}, two extra derivatives emerge. The subsequent lemma presents crucial inequalities for managing these additional derivatives.

First, let us fix some notation.
We denote by $D^p_x$ the differential operator given by $\widehat{D^p_xw}=|\xi|^p\widehat{w}$ and by  $a+$ the quantity $a+\varepsilon$ for any $\varepsilon>0$.

\begin{lemma}
\label{strichartz.lemma2}
For all $b>1/2$ and $0<T\lesssim 1$, one has
\begin{align}
\label{U-L2}
\| U\|_{L_x^2L_T^\infty}
&\le 
C\|U\|_{X^{\frac 34 +,b}_T}\\
\label{DxU}
\| D_x U\|_{L_x^\infty L_T^2}
&\le 
C\|U\|_{X^{0,b}_T}.
\end{align}
\end{lemma}
\begin{proof}
Using classical transfer arguments (see, for example,  Lemma $2.9$ in \cite{Tao-b}), the inequalities  \eqref{U-L2} and \eqref{DxU} are immediate consequences of \eqref{kpv-1} and \eqref{kpv-2}, respectively. 
\end{proof}

\begin{remark} It is worth noticing that the assumption $0<T\lesssim 1$ in Lemma \ref{strichartz.lemma2} doesn't affect our argument because we will be using estimates \eqref{U-L2} and \eqref{DxU} on bounded time intervals with fixed width. 
\end{remark}

\begin{lemma}[Lemma $3$ in \cite{DMT}]
\label{cosh-est}
Let $\xi = \sum\limits_{j=1}^p\xi_j$ for $\xi_j\in\rr$, where $p\ge 1$ is an integer.
Then
\begin{equation}
\Big|1- \cosh(|\xi|) \prod\limits_{j=1}^p \sech(|\xi_j|) \Big|
\le
2^p \sum\limits_{j \neq k=1}^p |\xi_j||\xi_k|.
\end{equation}
\end{lemma}

\begin{lemma}
\label{est.fU-mKdV}
Let $\sigma>0$ and $F$ be defined by
\begin{equation}
\label{fU-def-mKdV}
F(U)
:=
\frac {\mu}3 \partial_x\Big[
U^3-\cosh(\sigma |D_x|)\big((\sech(\sigma|D_x|) U)^3\big)
\Big].
\end{equation}
Then, there is some $\frac12<b<1$ such that
\begin{align}
\label{F-L2}
\|F(U)\|_{L^2_xL^2_t}
&\le
C \sigma^{2}\|U\|^3_{X^{2,b}},\\
\label{UF-L2}
\|U F(U)\|_{L^2_xL^2_t}
&\le
C \sigma^{2}\|U\|^4_{X^{2,b}},
\end{align}
for some constant $C>0$ independent of $\sigma$.
\end{lemma}
\begin{proof}
We begin by writing the following estimation for the Fourier transform of $F$, a result that can be readily obtained using Lemma \ref{cosh-est}
\begin{equation*}
\begin{split}
|\widehat{F(U)}(\xi,\tau)|
&\le
C|\xi|\int_\ast \big[1-\cosh(\sigma|\xi|) \prod\limits_{j=1}^3 \sech(\sigma |\xi_j|) \big]
|\widehat{U}(\xi_1,\tau_1)| |\widehat{U}(\xi_2,\tau_2)||\widehat{U}(\xi_3,\tau_3)|\\
&\le
C\sigma^2 |\xi|
\int_\ast (|\xi_1||\xi_2|+|\xi_1||\xi_3|+|\xi_2||\xi_3|)
|\widehat{U}(\xi_1,\tau_1)| |\widehat{U}(\xi_2,\tau_2)||\widehat{U}(\xi_3,\tau_3)|,
\end{split}
\end{equation*}
where $\int_\ast$ denotes the integral over the set 
$\{(\xi, \tau)\in \rr^2;  \xi=\xi_1+\xi_2+\xi_3\text{ and }\tau=\tau_1+\tau_2+\tau_3\}$.
Without loss of generality, employing symmetry, we can assume $|\xi_1|\le |\xi_2| \le |\xi_3|$, allowing us to derive the following estimate for the absolute value of $\widehat{F(U)}$
\begin{equation}
\label{F-est}
|\widehat{F(U)}(\xi,\tau)|
\le
C\sigma^2 \int_\ast |\xi_2||\xi_3|^2
|\widehat{U}(\xi_1,\tau_1)| |\widehat{U}(\xi_2,\tau_2)||\widehat{U}(\xi_3,\tau_3)|.
\end{equation}

To establish \eqref{F-L2}, we utilize Parseval's identity and apply estimate \eqref{F-est} to obtain
$$
\|F(U)\|_{L^2_xL^2_t}
=
\|\widehat{F(U)}\|_{L^2_\xi L^2_\tau}
\le
C\sigma^2 \| \widehat{W_1 W_2 W_3}\|_{L^2_\xi L^2_\tau}
=
C\sigma^2 \| W_1 W_2 W_3\|_{L^2_x L^2_t},
$$
where 
\begin{equation}
\label{w-def}
\widehat{W_1}=|\widehat{U}|, \;\;
\widehat{W_2}=|\widehat{D_x U}|,\;\;
\widehat{W_3}=|\widehat{D^2_xU}|.
\end{equation} 
Then, it follows from Lemma \ref{strichartz.lemma} that for all $b> 1/2$ 
\begin{equation}
\label{F-L2-X1}
\|F(U)\|_{L^2_x L^2_t}
\le
C \sigma^2 \|W_1\|_{X^{0,b}}\|W_2\|_{X^{0,b}}\|W_3\|_{X^{0,b}}
\le
C\sigma^2 \|U\|_{X^{0,b}}\|U\|_{X^{1,b}}\|U\|_{X^{2,b}},
\end{equation}
thus ensuring the validity of inequality \eqref{F-L2}.

Regarding \eqref{UF-L2}, we initially note that
\begin{equation*}
\begin{split}
|\widehat{UF(U)}(\tilde{\xi},\tilde{\tau})|
&\le
\int |\widehat{U}(\tilde{\xi}-\xi,\tilde{\tau}-\tau)| |\widehat{F(U)}(\xi,\tau)|d\xi d\tau \\
&\le
C\sigma^2 \int |\widehat{U}(\tilde{\xi}-\xi,\tilde{\tau}-\tau)|\Big(\int_\ast |\xi_2| |\xi_3|^2|\widehat{U}(\xi_1,\tau_1)| |\widehat{U}(\xi_2,\tau_2)||\widehat{U}(\xi_3,\tau_3)|\Big)d\xi d\tau
\end{split}
\end{equation*}
where we have employed the estimate \eqref{F-est}.

Therefore, by Parseval's identity, we derive
\begin{equation*}
\|U F(U)\|_{L^2_x L^2_t}
\le
C\sigma^2 \|W_1^2 W_2 W_3\|_{L^2_x L^2_t},
\end{equation*}
where $W_i$ for $i=1,2,3$ are as defined  in \eqref{w-def}.
Using again Lemma \ref{strichartz.lemma}, we obtain for all $b>1/2$
\begin{equation*}
\|U F(U)\|_{L^2_x L^2_t}
\le
C\sigma^2  \|W_1\|_{X^{0,b}}^2 \|W_2\|_{X^{0,b}} \|W_3\|_{X^{0,b}}
\le
C\sigma^2 \|U\|^2_{X^{0,b}} \|U\|_{X^{1,b}} \|U\|_{X^{2,b}},
\end{equation*}
thus confirming the validity of inequality \eqref{UF-L2}.
\end{proof}

%
%
%
%
%
%
\section{Almost conservation law - Proof of Theorem \ref{ACL-mKdV-thm}} 
\label{Sec-3}
In this section we will provide a proof of the almost conservation law that plays a fundamental role in the proof of the main result of this work.

\begin{proof}[Proof of Theorem \ref{ACL-mKdV-thm}] We start proving \eqref{ACL-mKdV}. Note that the expression for $A_{\sigma}(t)$ provided in \eqref{A} can be represented as
\begin{equation}\label{A-m1}
 A_\sigma(t)=\mathcal{I}_{0}(t) +\mathcal{I}_{1}(t)+\mathcal{I}_{2}(t),
 \end{equation}
 where $U:=\cosh(\sigma |D_x|)u$ and
\begin{equation}\label{I_012}
\begin{split}
\mathcal{I}_{0}(t)
&=
\int [U(x,t)]^2dx,\\
\mathcal{I}_{1}(t)
&=
\int [\p_x U(x,t)]^2 dx-\frac{\mu}{6}\int [U(x,t)]^4 dx,\\
\mathcal{I}_{2}(t)
&=
\int [\p_x^2U(x,t)]^2dx +\frac{1}{18}\int [U(x,t)]^6dx -\frac{5\mu}{3}\int [U(x,t)\p_xU(x,t)]^2dx.
\end{split}
\end{equation}

Moreover, by applying the operator $\cosh(\sigma |D_x|)$ to the mKdV equation \eqref{mKdV-IVP}, we obtain
\begin{equation}
\label{mKdV-U}
\p_tU+ \p_x^3 U +\mu U^2\p_x U
=
F(U),
\end{equation}
where $F(U)$ is given by \eqref{fU-def-mKdV}.

Now, by following the same steps presented in the proof of Proposition~$3.1$ in \cite{FP24} (more precisely, see equation $(3.21)$ there), it can easily be shown that
\begin{equation}
\label{td-A}
\begin{split}
\frac{d}{dt}A_\sigma(t) &= 2\int UF(U)dx +2\int \p_xU\p_x[F(U)]dx -\frac {2\mu}{3} \int U^3F(U) dx+2\int \p^2_x U \p_x^2 [F(U)]dx\\
& 
\quad 
+
\frac13\int U^5F(U)dx
+\frac{10\mu}{3}\int U(\p_xU)^2F(U)dx
+\frac{10\mu}{3}\int U^2\p_x^2UF(U)dx.
\end{split}
\end{equation}

Now, integrating \eqref{td-A} over the time interval $[0,t]$ with $0<t<T$, we obtain
\begin{equation}
\label{goal.A}
A_\sigma(t)
=
A_\sigma(0) + R(t),
\end{equation}
where
\begin{equation}
\label{rt-1}
R(t) = R_{1}(t)+ R_{2}(t)+R_{3}(t)+R_4(t),
\end{equation}
with
\begin{align}
\label{r1t}
R_{1}(t)
&=
2\iint \chi_{[0,t]}UF(U)dxdt', \\
\label{r2t}
R_{2}(t)
&=
-2\iint \chi_{[0,t]}\p_x^2 U F(U)dxdt'
-\frac{2\mu}{3}\iint \chi_{[0,t]} U^3F(U)dxdt',\\
\label{r3t}
R_{3}(t)
&=
\frac13\iint  \chi_{[0,t]}U^5F(U)dxdt' 
+\frac{10\mu}{3}\bigg(\iint  \chi_{[0,t]}[U(\p_xU)^2F(U)
+U^2\p_x^2UF(U)]dxdt'\bigg),\\
R_{4}(t)
&=
2\iint  \chi_{[0,t]}\p^2_x U \p_x^2 [F(U)]dx dt'.
\end{align}

In view of \eqref{goal.A}, the estimate \eqref{ACL-mKdV} will follow if we can show that
\begin{equation}
\label{goal.r}
|R(t)|
\le
C\sigma^2 \|u\|^4_{X^{\sigma,2,b}_T} \big(1+\|u\|^2_{X^{\sigma,2,b}_T}+\|u\|^4_{X^{\sigma,2,b}_T}\big),
\end{equation}
for all $t\in [0,T]$.

Now, applying duality we obtain
\begin{align*}
|R_{1}(t)|
&\le
C\| \chi_{[0,T]}U\|_{L^2_xL^2_t} \|\chi_{[0,T]}F(U)\|_{L^2_xL^2_t}, \\
|R_{2}(t)|
&\le
C\Big(\| \chi_{[0,T]}\p_x^2U\|_{L^2_xL^2_t} \|\chi_{[0,T]}F(U)\|_{L^2_xL^2_t}
+
\| \chi_{[0,T]}U^3\|_{L^2_xL^2_t} \|\chi_{[0,T]}F(U)\|_{L^2_xL^2_t}\Big),\\
|R_{3}(t)|
&\le
C\Big(
\|  \chi_{[0,T]}U^4\|_{L^2_xL^2_t}\|\chi_{[0,T]}UF(U)\|_{L^2_xL^2_t}
+
\| \chi_{[0,T]}U(\p_xU)^2\|_{L^2_xL^2_t}\|\chi_{[0,T]}F(U)\|_{L^2_xL^2_t} \nonumber\\
&\qquad\;\;+\|\chi_{[0,T]} U^2\p_x^2U\|_{L^2_xL^2_t}\|\chi_{[0,T]}F(U)\|_{L^2_xL^2_t}
\Big).
\end{align*}
Then, by combining Lemmas \ref{est-XT}, \ref{strichartz.lemma} and \ref{est.fU-mKdV}, with the estimates restricted in time slab, we derive
\begin{align}\label{est-r}
\nonumber
|R_{1}(t)|
&\le
C\sigma^2 \| U\|^4_{X_T^{2,b}},  \\
|R_{2}(t)|
&\le
C\sigma^2\Big( \| U\|^4_{X_T^{2,b}}
+
 \| U\|^6_{X_T^{2,b}}\Big),\\
 \nonumber
|R_{3}(t)|
&\le
C\sigma^2 \Big(
 \| U\|^6_{X_T^{2,b}}
 +  
 \| U\|^8_{X_T^{2,b}}\Big).
\end{align}

In this way, to get  the desired inequality \eqref{goal.r}, the only remaining task is to estimate the absolute value of the term $R_4(t)$ so as to complete the proof of \eqref{ACL-mKdV}.
To achieve this, employing the same steps to derive \eqref{F-est}, we observe that the Fourier transform of $F$ restricted to spatial variable can be bounded as follows
\begin{equation}
\label{F-est-x}
|\mathcal{F}_x(F(U))(\xi,t')|
\le
C\sigma^2 \int_\bullet |\xi_2||\xi_3|^2
|\mathcal{F}_x(U)(\xi_1,t')| |\mathcal{F}_x(U)(\xi_2,t')||\mathcal{F}_x(U)(\xi_3,t')|, 
\end{equation}
if we assume $|\xi_1|\le |\xi_2| \le |\xi_3|$, where $\int_\bullet$ denotes that the integral is taken over the set $\{\xi\in\rr; \;\xi=\xi_1+\xi_2+\xi_3\}$. 

Applying Parseval's identity in $x$ variable and estimate \eqref{F-est-x}, we get
\begin{equation*}
\begin{split}
|R_4(t)|
&\le
C \sigma^2
\int \chi_{[0,t]}\bigg(\int_\bullet |\xi|^4 |\mathcal{F}_x(U)(\xi,t')| |\xi_2| |\xi_3|^2
|\mathcal{F}_x(U)(\xi_1,t')| |\mathcal{F}_x(U)(\xi_2,t')||\mathcal{F}_x(U)(\xi_3,t')| \bigg)dt'\\
&\le
C \sigma^2
\int \chi_{[0,t]}\bigg(\int_\bullet |\xi|^3 |\mathcal{F}_x(U)(\xi,t')| |\xi_2| |\xi_3|^3
|\mathcal{F}_x(U)(\xi_1,t')| |\mathcal{F}_x(U)(\xi_2,t')||\mathcal{F}_x(U)(\xi_3,t')| \bigg)dt'\\
&=
C\sigma^2
\int \chi_{[0,t]}\bigg(\int \Big[\overline{\mathcal{F}_x(D_x^3W)}(\xi,t')\Big]
\Big[\mathcal{F}_x( W\cdot D_xW \cdot D_x^3 W)(\xi,t')\Big] d\xi\bigg) dt'
\\
&=
C\sigma^2 
\iint \chi_{[0,t]} \overline{D_x^3W} \cdot W\cdot D_xW \cdot D_x^3 W dxdt',
\end{split}
\end{equation*}
since $|\xi|\le 3|\xi_3|$, where $\mathcal{F}_xW=|\mathcal{F}_x(U)|$.

Now, it follows from Holder's inequality and estimates \eqref{U-L2} and \eqref{DxU} that
\begin{equation}
\label{est-r4}
\begin{split}
|R_4(t)|
&\le
C\sigma^2 
\|D_x^3W\|_{L_x^\infty L^2_T}^2  \|W\|_{L_x^2 L^\infty_T} \|D_xW\|_{L_x^2 L^\infty_T}\\
&\le
C\sigma^2
\|D_x^2W\|_{X_{T}^{0, b}}^2\|W\|_{X_{T}^{\frac 34 +, b}} \|D_xW\|_{X_{T}^{\frac 34 +, b}}\\
&\le
C\sigma^2
\|U\|^2_{X_{T}^{2, b}}\|U\|_{X_{T}^{\frac 34 +, b}} \|U\|_{X_{T}^{\frac 74 +, b}}\\
&\le
C\sigma^2 \|U\|_{X_{T}^{2,b}}^4.
\end{split}
\end{equation}

  Using the estimates in \eqref{est-r} and \eqref{est-r4} it readily follows that $R_{\sigma}(t)$ satisfies \eqref{goal.r}, since we have
\begin{equation}
\label{U-u}
\|U\|_{X_T^{2,b}} \le  \|u\|_{X_T^{\sigma,2,b}}.
\end{equation}
In fact, let $v\in X^{\sigma,2,b}$ be an extension of $u$ in $[-T,T]\times \rr$, which implies that $V=\cosh(\sigma |D_x|)v$ is an extension of $U$.
Then, using the definition of the norm $\|\cdot\|_{X^{\sigma,s,b}_T}$  we get 
\begin{equation*}
\|U\|_{X_T^{2,b}}
\le
\|V\|_{X^{2,b}}
=
\|\la \xi\ra^2 \la\tau-\xi^3\ra^b \cosh(\sigma |\xi|)\widehat{v}(\xi,\tau) \|_{L^2_\xi L^2_\tau}
=
\|v\|_{X^{\sigma,2,b}},
\end{equation*}
for all extension $v$ of $u$, thus proving \eqref{U-u} and concluding the proof of \eqref{ACL-mKdV}.

In what follows, we will supply a  proof of \eqref{ACL-mKdV-0}. Using the bound \eqref{bound.u}  of the local solution in  the almost conservation law \eqref{ACL-mKdV}, we get 
\begin{equation}\label{est-At}
A_\sigma(t)
\le
A_\sigma(0) +C\sigma^2 \|u_0\|_{H^{\sigma,2}}^4(1+\|u_0\|_{H^{\sigma,2}}^2+\|u_0\|_{H^{\sigma,2}}^4).
\end{equation}
Hence, the proof of \eqref{ACL-mKdV-0} follows from \eqref{est-At} by noticing that 
\begin{equation}
\label{defocusing}
 \|u_0\|^2_{H^{\sigma,s}}
\le  
A_{\sigma}(0), \text{ when $\mu=-1$}.
\end{equation}
This finishes the proof of Theorem \ref{ACL-mKdV-thm}.
\end{proof}

%
%
%
%
%
%
\section{Global analytic solution - Proof of Theorem \ref{global-mKdV-thm} }
\label{Sec-4}

In this section, we use the almost conservation low established in the previous section and provide a proof of the main result of this work. The idea of the proof is similar to the one used  in \cite{FP23} and \cite{FP24}.

\begin{proof}[Proof of Theorem \ref{global-mKdV-thm}]
The proof  follows precisely the same steps as presented in the proof of Theorem $1.4$ in \cite{FP23}, except for the fact that now we fix $s=2$ and rely on the exponent $2$ in the radius of analyticity $\sigma$ that appears in the almost conservation law \eqref{ACL-mKdV-0}. This change produces $T^{-\frac 12}$ as a lower bound for the radius of analyticity, that is, we are able to extend the local solution $u$ to the time interval $[0, T]$  for any $T> T_0$, thereby obtaining
$$
u\in C([-T,T]; H^{\sigma(T),s}(\rr)),
\text{ with }
\sigma(T)\ge \min\Big\{\sigma_0, \frac{1}{T^{\frac 12}}\Big\},
$$
as stated in Theorem \ref{global-mKdV-thm}. As the procedure is the same as the one used in \cite{FP23}, we omit the details.
\end{proof}


\vskip 0.3cm
\noindent{\bf Acknowledgements.} \\
 The first author acknowledges the support from FAPESP, Brazil (\#2021/04999-9).
The second author acknowledges the grants from FAPESP, Brazil (\#2023/06416-6) and CNPq, Brazil (\#307790/2020-7). 
\\


\noindent
{\bf Conflict of interest statement.} \\
On behalf of all authors, the corresponding author states that there is no conflict of interest.\\

\noindent 
{\bf Data availability statement.} \\
The datasets generated during and/or analysed during the current study are available from the corresponding author on reasonable request.



%

\end{document}